\documentclass{amsproc}
\usepackage{amsmath}
\usepackage{enumerate}
\usepackage{amsmath,amsthm,amscd,amssymb}

\usepackage{latexsym}
\usepackage{upref}
\usepackage{verbatim}

\usepackage[mathscr]{eucal}
\usepackage{dsfont}

\usepackage{graphicx}
\usepackage[colorlinks,hyperindex,hypertex]{hyperref}

\usepackage{hhline}

\usepackage[OT2,OT1]{fontenc}
\newcommand\cyr
{
\renewcommand\rmdefault{wncyr}
\renewcommand\sfdefault{wncyss}
\renewcommand\encodingdefault{OT2}
\normalfont
\selectfont
}
\DeclareTextFontCommand{\textcyr}{\cyr}

\newtheorem{theorem}{Theorem}

\newtheorem{definition}[theorem]{Definition}

\chardef\bslash=`\\ 

\newcommand{\bbR}{{\mathbb{R}}}

\newcommand{\bbC}{{\mathbb{C}}}

\newcommand{\ran}{\text{\rm{Ran}}}

\newcommand{\tr}{\text{\rm{tr}}}

\newcommand{\calA}{{\mathcal A}}

\newcommand{\calD}{{\mathcal D}}

\newcommand{\calH}{{\mathcal H}}

\newcommand{\calN}{{\mathcal N}}

\newcommand{\calS}{{\mathcal S}}

\def\sM{{\mathfrak M}}

      \def\dC{{\mathbb C}}

   \def\cN{{\mathcal N}}   
      
\def\cS{{\mathcal S}}

\def\RE{{\rm Re\,}}

\DeclareMathOperator{\IM}{Im}

\newcommand{\eval}[2][\right]{\relax
  \ifx#1\right\relax \left.\fi#2#1\rvert}

\begin{document}

\title{The c-Entropy of non-dissipative L-systems}

\author{S. Belyi}
\address{Department of Mathematics\\ Troy University\\
Troy, AL 36082, USA\\
}
\curraddr{}
\email{sbelyi@troy.edu}


\author[Makarov]{K. A. Makarov}
\address{Department of Mathematics\\
 University of Missouri\\
  Columbia, MO 63211, USA}
\email{makarovk@missouri.edu}


\author{E. Tsekanovskii}
\address{Department of Mathematics, Niagara University,  Lewiston,
NY  14109, USA} \email{\tt tsekanov@niagara.edu}


\subjclass{Primary 47A10; Secondary 47N50, 81Q10}
\date{DD/MM/2004}


\keywords{L-system, transfer function, impedance function,  Herglotz-Nevan\-linna function,  c-entropy, dissipation coefficient, accumulation coefficient}

\begin{abstract}
In this paper, we extend the definition of c-entropy to canonical L-systems with non-dissipative state-space operators. We also introduce the concepts of  dissipation and accumulation coefficients for such systems. In addition,  we examine the coupling of these L-systems and derive closed form expressions  for the corresponding c-entropy.
 \end{abstract}

\maketitle

\tableofcontents


\section{Introduction}\label{s1}

In this paper we continue our  study of L-systems' c-entropy (see \cite{BT-21, BT-22, BT-23, BT-16}) and its properties.

Recall the concept of a canonical L-system.

Let $A$ be a bounded linear operator in a Hilbert space $\calH$ and $E$ is another Hilbert space with $\dim E<\infty$.
By  a  \textit{canonical L-system} we mean the array
\begin{equation}
\label{col0}
 \Theta =
\left(%
\begin{array}{ccc}
  A    & K & J \\
   \calH&  & E \\
\end{array}%
\right),
\end{equation}
where $K\in[E,\calH]$,  $J$ is a bounded, self-adjoint, and unitary operator in $E$, and   $\IM A=KJK^*$.
The operator-valued function
\begin{equation*}\label{W1}
 W_\Theta(z)=I-2iK^*(A-zI)^{-1}KJ,\quad z\in \rho(A),
\end{equation*}
 is called the \textit{transfer function}  of an L-system $\Theta$ and
\begin{equation*}\label{real2}
 V_\Theta(z)=i[W_\Theta(z)+I]^{-1}[W_\Theta(z)-I]J =K^*(\RE A-zI)^{-1}K,\quad z\in\rho(A)\cap\dC_{\pm},
\end{equation*}
is called the \textit{impedance function } of $\Theta$. The formal definition of  L-systems is presented in Section \ref{s2}.

The main goal of this note is to  extend  the concept of  c-Entropy  to canonical non-dissipative L-systems with two-dimensional input-output space $E$. We are going to utilize simple yet descriptive model L-systems based on the multiplication operator covered in \cite{BT-23}. In addition to that we will generalize the notion of the dissipation coefficient  introduced in \cite{BT-16} and \cite{BT-21} (see also \cite{BT-23}).

The paper is organized as follows.

Section \ref{s2} contains necessary information on the L-systems theory.

In Section \ref{s3} we present  a detailed construction of a canonical dissipative L-system associated with the multiplication by a scalar operator. The input-output space for such an L-system is $\dC^2$ and both transfer and impedance functions  are $(2\times 2)$ matrix-valued functions.

Sections \ref{s4}--\ref{s7} of the paper contain the  main results.

Section \ref{s4} provides   a construction of  canonical L-systems that are similar to those in Section \ref{s3}. These  systems
are no longer dissipative and have the directing operator $J$ different from the identity matrix. We also provide  explicit formulas for the transfer and impedance matrix-functions.

In Section \ref{s5} we extend the definition  of c-entropy to a wider class of canonical L-systems compared those discussed in \cite{BT-21, BT-22, BT-23, BT-16}. We apply this definition to L-systems constructed in Sections \ref{s3} and \ref{s4} and obtain explicit representations for  c-entropy in each case. In addition to that, we determine the model parameter(s)  that  yield(s) an extreme c-entropy value.

In Section \ref{s6}, we extend the definition of the dissipation coefficient initially proposed in \cite{BT-16, BT-21,BT-23} to non-dissipative systems. This generalization introduces a novel concept of accumulation coefficient previously unexplored in this context. Additionally, we derive explicit analytical expressions for both dissipation and accumulation coefficients within the model L-systems framework developed in Sections \ref{s3} and \ref{s4}.

In Section \ref{s7}, we revisit the concept of L-system coupling  (see \cite{ABT,Bro}), and\ derive the explicit form of the coupling between two canonical L-systems discussed   Sections \ref{s3} and \ref{s4} followed by an independent proof of  the Multiplication Theorem  (see, e.g., \cite{Bro,MT10}): the transfer function of the  coupling coincides with  the product of the transfer functions associated with the factor L-systems.

The paper is concluded with illustrative examples that demonstrate the constructions and concepts discussed.

\section{Preliminaries}\label{s2}

For a pair of Hilbert spaces $\calH_1$, $\calH_2$ denote by $[\calH_1,\calH_2]$ the set of all bounded linear operators from $\calH_1$ to $\calH_2$.

Let $T$ be a bounded linear operator in a Hilbert space $\calH$, $K\in[E,\calH]$, and $J$ be a bounded, self-adjoint, and unitary operator in $E$, where $E$ is another Hilbert space with $\dim E<\infty$. Let also
\begin{equation}\label{e5-3}
\IM T=KJK^*.
\end{equation}

By definition, the  \textbf{Liv\v{s}ic canonical
system } or simply \textbf{canonical L-system} is just the  array
\begin{equation}
\label{BL}
 \Theta =
\left(%
\begin{array}{ccc}
  T & K & J \\
  \calH &  & E \\
\end{array}%
\right).
 \end{equation}
 The spaces $\calH$ and $E$ here  are
called \textit{state} and
\textit{input-output} spaces, and the operators
$T$, $K$, $J$ will be refered to as \textit{main},
\textit{channel}, and \textit{directing} operators, respectively.

Notice that
relation
\eqref{e5-3}
implies
\begin{equation}\label{e5-4}
    \ran (\IM T)\subseteq\ran(K).
\end{equation}

We  associate with an L-system $\Theta$ two  analytic functions,  the \textbf{transfer  function} of the L-system $\Theta$
\begin{equation}\label{e6-3-3}
W_\Theta (z)=I-2iK^\ast (T-zI)^{-1}KJ,\quad z\in \rho (T),
\end{equation}
and also the \textbf{impedance function}  given by the formula
\begin{equation}\label{e6-3-5}
V_\Theta (z) = K^\ast (\RE T - zI)^{-1} K, \quad z\in  \rho (\RE T).
\end{equation}

 The transfer function $W_\Theta (z)$ of the L-system $\Theta $ and function $V_\Theta (z)$ of the form (\ref{e6-3-5}) are connected by the following relations valid for $\IM z\ne0$, $z\in\rho(T)$,
\begin{equation}\label{e6-3-6}
\begin{aligned}
V_\Theta (z) &= i [W_\Theta (z) + I]^{-1} [W_\Theta (z) - I]J,\\
W_\Theta(z)&=(I+iV_\Theta(z)J)^{-1}(I-iV_\Theta(z)J).
\end{aligned}
\end{equation}

Recall that  the  impedance function $V_\Theta(z)$ of a canonical L-system admits the  integral representation (see, e.g.,  \cite[Section 5.5]{ABT}, \cite{Bro})
\begin{equation}\label{hernev-real}
V_\Theta(z)=\int_\bbR \frac{d\sigma(t)}{{t}-z},
\end{equation}
where  $\sigma$ is an   operator-valued bounded Borel measure in $E$ with the compact support on $\bbR$.

As far as the inverse problem is concerned,  we refer to  \cite{ABT,BMkT,GT} and references therein for the description
of the class of all Herglotz-Nevanlinna functions that admit  realizations as impedance  functions of an L-system.

\section{An L-system with a dissipative operator}\label{s3}

Let $\calH^2$ be a two-dimensional Hilbert space with an inner product $(\cdot,\cdot)$ and an orthogonal normalized basis of vectors $h_{01},h_{02}\in \calH^2$, ($\|h_{01}\|=\|h_{02}\|=1$). For a fixed number $\lambda_0\in\dC$ such that $\IM\lambda_0>0$ we introduce (see also \cite{Bro}) a linear operator
\begin{equation}\label{e4-22-d}
    T_d h=\left[
          \begin{array}{cc}
            \lambda_0 & 0 \\
            0 & -\bar\lambda_0 \\
          \end{array}
        \right]
     h,\quad h=\left[
                               \begin{array}{c}
                                 h_1 \\
                                 h_2 \\
                               \end{array}
                             \right]   \in\calH^2.
\end{equation}
Clearly,
$$
T^*_d=\left[
          \begin{array}{cc}
            \bar\lambda_0 & 0 \\
            0 & -\lambda_0 \\
          \end{array}
        \right],\;\IM T_d =\left[
          \begin{array}{cc}
            \IM\lambda_0 & 0 \\
            0 & \IM\lambda_0 \\
          \end{array}
        \right],\; \RE T_d =\left[
          \begin{array}{cc}
            \RE\lambda_0 & 0 \\
            0 & -\RE\lambda_0 \\
          \end{array}
        \right].
$$
 We are going to include $T_d$ into an  L-system $\Theta$. Let $K:\dC^2\rightarrow\calH^2$ be such that
\begin{equation}\label{e4-23-d}
   K \left[
                               \begin{array}{c}
                                 c_1 \\
                                 c_2 \\
                               \end{array}
                             \right]=(\sqrt{\IM\lambda_0})\left[
          \begin{array}{cc}
            h_{01} & 0 \\
            0 & h_{02} \\
          \end{array}
        \right]\left[
                               \begin{array}{c}
                                 c_1 \\
                                 c_2 \\
                               \end{array}
                             \right]=(\sqrt{\IM\lambda_0})\left[
                               \begin{array}{c}
                               c_1   h_{01} \\
                                 c_2  h_{02}\\
                               \end{array}
                             \right]\in\calH^2,
\end{equation}
for all $c_1,\,c_2\in\dC$. Then, the adjoint operator $K^*:\calH^2\rightarrow\dC^2$ acts on an arbitrary vector
$h=\left[
                               \begin{array}{c}
                                 h_1 \\
                                 h_2 \\
                               \end{array}
                             \right] =\left[
                               \begin{array}{c}
                               c_1   h_{01} \\
                                 c_2  h_{02}\\
                               \end{array}
                             \right]\in\calH^2$ as follows
\begin{equation}\label{e-4-32}
   K^* h=K^*\left[
                               \begin{array}{c}
                               c_1   h_{01} \\
                                 c_2  h_{02}\\
                               \end{array}
                             \right]=\sqrt{\IM\lambda_0}\left[
                               \begin{array}{c}
                               (h,c_1   h_{01}) \\
                                (h, c_2  h_{02})\\
                               \end{array}
                             \right]=\sqrt{\IM\lambda_0}\left[
                               \begin{array}{c}
                                 c_1 \\
                                 c_2 \\
                               \end{array}
                             \right].
\end{equation}
Let
\begin{equation}\label{e25-I}
    I=\left[
          \begin{array}{cc}
            1 & 0 \\
            0 & 1 \\
          \end{array}
        \right]
\end{equation}
be the identity operator.
Furthermore,
\begin{equation}\label{e4-26-d}
\begin{aligned}
KIK^* h&=\sqrt{\IM\lambda_0}\left[
          \begin{array}{cc}
            h_{01} & 0 \\
            0 & h_{02} \\
          \end{array}
        \right]\left[
          \begin{array}{cc}
            1 & 0 \\
            0 & 1 \\
          \end{array}
        \right]\sqrt{\IM\lambda_0}\left[
                               \begin{array}{c}
                                 c_1 \\
                                 c_2 \\
                               \end{array}
                             \right]\\
                             &=\left[
          \begin{array}{cc}
            \IM\lambda_0 & 0 \\
            0 & \IM\lambda_0 \\
          \end{array}
        \right]\left[
                               \begin{array}{c}
                                 h_1 \\
                                 h_2 \\
                               \end{array}
                             \right] =\IM T_d h.
\end{aligned}
\end{equation}
Thus, we can construct an L-system of the form
\begin{equation}\label{e4-27-d}
    \Theta_d= \begin{pmatrix} T_d&K&\ I\cr  \calH^2 & &\dC^2\cr \end{pmatrix},
\end{equation}
where operators $T_d$ and $K$  are defined by \eqref{e4-22-d} and  \eqref{e4-23-d}, respectively. Taking into account that
$$
(T_d-zI)^{-1}=\left[
          \begin{array}{cc}
            \frac{1}{\lambda_0-zI} & 0 \\
            0 & -\frac{1}{\overline{\lambda_0}+zI} \\
          \end{array}
        \right],
$$
we proceed with calculations of the transfer function $W_\Theta(z)$.

{We have
\begin{equation}\label{e4-28-d}
    \begin{aligned}
    W_{\Theta_d} (z)&=I-2iK^\ast (T_d-zI)^{-1}KI
    =\left[
          \begin{array}{cc}
            1 & 0 \\
            0 & 1 \\
          \end{array}
        \right]-2i\IM \lambda_0\left[
          \begin{array}{cc}
            \frac{1}{\lambda_0-z} & 0 \\
            0 & -\frac{1}{\overline{\lambda_0}+z} \\
          \end{array}
        \right]\\
    &=\left[
          \begin{array}{cc}
            1 & 0 \\
            0 & 1 \\
          \end{array}
        \right]-\left[
          \begin{array}{cc}
            \frac{\lambda_0-\bar\lambda_0}{\lambda_0-z} & 0 \\
            0 & -\frac{\lambda_0-\bar\lambda_0}{\overline{\lambda_0}+z} \\
          \end{array}
        \right]=\left[
          \begin{array}{cc}
            \frac{\bar\lambda_0-z}{\lambda_0-z} & 0 \\
            0 & \frac{\lambda_0+z}{{\bar\lambda_0}+z} \\
          \end{array}
        \right].
    \end{aligned}
\end{equation}
}
The corresponding impedance function of the form \eqref{e6-3-5} is easily found and given by
\begin{equation}\label{e4-29-d}
    V_{\Theta_d} (z)=K^\ast (\RE T_d - zI)^{-1} K=\left[
          \begin{array}{cc}
            \frac{\IM\lambda_0}{\RE\lambda_0-z} & 0 \\
            0 & -\frac{\IM\lambda_0}{\RE\lambda_0+z} \\
          \end{array}
        \right].
\end{equation}
By direct check one confirms that $V_\Theta (z)$ is a Herglotz-Nevanlinna matrix-valued function.

\section{L-systems with a non-dissipative operator}\label{s4}

Let $\calH^2$ be a two-dimensional Hilbert space with an inner product $(\cdot,\cdot)$ and an orthogonal normalized basis of vectors $h_{01},h_{02}\in \calH^2$, ($\|h_{01}\|=\|h_{02}\|=1$). For a fixed number $\lambda_0\in\dC$ such that $\IM\lambda_0>0$ we introduce (see also \cite{Bro}) a linear operator
\begin{equation}\label{e-d-4-30}
    T_m h=\left[
          \begin{array}{cc}
            \lambda_0 & 0 \\
            0 & \bar\lambda_0 \\
          \end{array}
        \right]
     h,\quad h=\left[
                               \begin{array}{c}
                                 h_1 \\
                                 h_2 \\
                               \end{array}
                             \right]   \in\calH^2.
\end{equation}
Clearly,
$$
T^*_m=\left[
          \begin{array}{cc}
            \bar\lambda_0 & 0 \\
            0 & \lambda_0 \\
          \end{array}
        \right],\;\IM T_m =\left[
          \begin{array}{cc}
            \IM\lambda_0 & 0 \\
            0 & -\IM\lambda_0 \\
          \end{array}
        \right],\; \RE T_m =\left[
          \begin{array}{cc}
            \RE\lambda_0 & 0 \\
            0 & \RE\lambda_0 \\
          \end{array}
        \right].
$$
 We are going to include $T_m$ into an  L-system $\Theta_m$. Let $K:\dC^2\rightarrow\calH^2$ be  of the form \eqref{e4-23-d},  that is,
\begin{equation*}\label{e-4-31-m}
   K \left[
                               \begin{array}{c}
                                 c_1 \\
                                 c_2 \\
                               \end{array}
                             \right]=(\sqrt{\IM\lambda_0})\left[
          \begin{array}{cc}
            h_{01} & 0 \\
            0 & h_{02} \\
          \end{array}
        \right]\left[
                               \begin{array}{c}
                                 c_1 \\
                                 c_2 \\
                               \end{array}
                             \right]=(\sqrt{\IM\lambda_0})\left[
                               \begin{array}{c}
                               c_1   h_{01} \\
                                 c_2  h_{02}\\
                               \end{array}
                             \right]\in\calH^2,
\end{equation*}
for all $c_1,\,c_2\in\dC$. Then, the adjoint operator $K^*:\calH^2\rightarrow\dC^2$ is of the form \eqref{e-4-32} and acts on an arbitrary vector
$h=\left[
                               \begin{array}{c}
                                 h_1 \\
                                 h_2 \\
                               \end{array}
                             \right] =\left[
                               \begin{array}{c}
                               c_1   h_{01} \\
                                 c_2  h_{02}\\
                               \end{array}
                             \right]\in\calH^2$ as follows
\begin{equation*}\label{e-4-19}
   K^* h=K^*\left[
                               \begin{array}{c}
                               c_1   h_{01} \\
                                 c_2  h_{02}\\
                               \end{array}
                             \right]=\sqrt{\IM\lambda_0}\left[
                               \begin{array}{c}
                               (h,c_1   h_{01}) \\
                                (h, c_2  h_{02})\\
                               \end{array}
                             \right]=\sqrt{\IM\lambda_0}\left[
                               \begin{array}{c}
                                 c_1 \\
                                 c_2 \\
                               \end{array}
                             \right].
\end{equation*}

Define a signature operator
\begin{equation}\label{e-25-J}
    J_m=\left[
          \begin{array}{cc}
            1 & 0 \\
            0 & -1 \\
          \end{array}
        \right]
\end{equation}
and note that $J_m=J^*_m=J^{-1}_m$.
Furthermore,
\begin{equation}\label{e-4-33}
\begin{aligned}
KJ_m K^* h&=\sqrt{\IM\lambda_0}\left[
          \begin{array}{cc}
            h_{01} & 0 \\
            0 & h_{02} \\
          \end{array}
        \right]\left[
          \begin{array}{cc}
            1 & 0 \\
            0 & -1 \\
          \end{array}
        \right]\sqrt{\IM\lambda_0}\left[
                               \begin{array}{c}
                                 c_1 \\
                                 c_2 \\
                               \end{array}
                             \right]\\
                             &=\left[
          \begin{array}{cc}
            \IM\lambda_0 & 0 \\
            0 & -\IM\lambda_0 \\
          \end{array}
        \right]\left[
                               \begin{array}{c}
                                 h_1 \\
                                 h_2 \\
                               \end{array}
                             \right] =\IM T_m h.
\end{aligned}
\end{equation}
Thus, we can construct an L-system of the form
\begin{equation}\label{e-4-34}
    \Theta_m= \begin{pmatrix} T_m&K&\ J_m\cr  \calH^2 & &\dC^2\cr \end{pmatrix},
\end{equation}
where operators $T_m$, $K$,  and $J_m$ are defined by \eqref{e-d-4-30},  \eqref{e4-23-d}, and \eqref{e-25-J}, respectively. Taking into account that
$$
(T_m-zI)^{-1}=\left[
          \begin{array}{cc}
            \frac{1}{\lambda_0-zI} & 0 \\
            0 & \frac{1}{\overline{\lambda_0}-zI} \\
          \end{array}
        \right],
$$
we proceed with calculations of the transfer function $W_{\Theta_m}(z)$. We have
\begin{equation}\label{e-4-35}
    \begin{aligned}
    W_{\Theta_m} (z)&=I-2iK^\ast (T_m-zI)^{-1}KJ_m=\left[
          \begin{array}{cc}
            1 & 0 \\
            0 & 1 \\
          \end{array}
        \right]-2i\IM \lambda_0\left[
          \begin{array}{cc}
            \frac{1}{\lambda_0-z} & 0 \\
            0 & -\frac{1}{\overline{\lambda_0}-z} \\
          \end{array}
        \right]\\
    &=\left[
          \begin{array}{cc}
            1 & 0 \\
            0 & 1 \\
          \end{array}
        \right]-\left[
          \begin{array}{cc}
            \frac{\lambda_0-\bar\lambda_0}{\lambda_0-z} & 0 \\
            0 & -\frac{\lambda_0-\bar\lambda_0}{\overline{\lambda_0}-z} \\
          \end{array}
        \right]=\left[
          \begin{array}{cc}
            \frac{\bar\lambda_0-z}{\lambda_0-z} & 0 \\
            0 & \frac{\lambda_0-z}{\overline{\lambda_0}-z} \\
          \end{array}
        \right].
    \end{aligned}
\end{equation}
The corresponding impedance function of the form \eqref{e6-3-5} is easily found and given by
\begin{equation}\label{e-4-36}
    V_{\Theta_m} (z)=K^\ast (\RE T_m - zI)^{-1} K=\left[
          \begin{array}{cc}
            \frac{\IM\lambda_0}{\RE\lambda_0-z} & 0 \\
            0 & \frac{\IM\lambda_0}{\RE\lambda_0-z} \\
          \end{array}
        \right]=\frac{\IM\lambda_0}{\RE\lambda_0-z}I.
\end{equation}
By direct check one confirms that $V_{\Theta_m} (z)$ is a Herglotz-Nevanlinna matrix-valued function.

Now we are going to construct yet another L-system whose main operator is not dissipative. Using the same conventions as in the beginning of the section we set
\begin{equation}\label{e-d-4-33}
    T_a h=\left[
          \begin{array}{cc}
            \bar\lambda_0 & 0 \\
            0 & -\lambda_0 \\
          \end{array}
        \right]
     h,\quad h=\left[
                               \begin{array}{c}
                                 h_1 \\
                                 h_2 \\
                               \end{array}
                             \right]   \in\calH^2.
\end{equation}
Clearly,
$$
T^*_a=\left[
          \begin{array}{cc}
            \lambda_0 & 0 \\
            0 & -\bar\lambda_0 \\
          \end{array}
        \right],\;
        \IM T_a =\left[
          \begin{array}{cc}
            -\IM\lambda_0 & 0 \\
            0 & -\IM\lambda_0 \\
          \end{array}
        \right],\; \RE T_a =\left[
          \begin{array}{cc}
            \RE\lambda_0 & 0 \\
            0 & -\RE\lambda_0 \\
          \end{array}
        \right].
$$
 We are going to include $T_a$ into an  L-system $\Theta$. Let $K:\dC^2\rightarrow\calH^2$ be still defined via \eqref{e4-23-d} with $K^*:\calH^2\rightarrow\dC^2$ given by \eqref{e-4-32}.

Define a signature operator
\begin{equation}\label{e-34-J}
    J_a=-I=\left[
          \begin{array}{cc}
            -1 & 0 \\
            0 & -1 \\
          \end{array}
        \right]
\end{equation}
and note that $J_a=J^*_a=J^{-1}_a$.
Furthermore,
\begin{equation}\label{e-4-35-KJK}
\begin{aligned}
KJ_a K^* h&=\sqrt{\IM\lambda_0}\left[
          \begin{array}{cc}
            h_{01} & 0 \\
            0 & h_{02} \\
          \end{array}
        \right]\left[
          \begin{array}{cc}
            -1 & 0 \\
            0 & -1 \\
          \end{array}
        \right]\sqrt{\IM\lambda_0}\left[
                               \begin{array}{c}
                                 c_1 \\
                                 c_2 \\
                               \end{array}
                             \right]\\
                             &=\left[
          \begin{array}{cc}
            -\IM\lambda_0 & 0 \\
            0 & -\IM\lambda_0 \\
          \end{array}
        \right]\left[
                               \begin{array}{c}
                                 h_1 \\
                                 h_2 \\
                               \end{array}
                             \right] =\IM T_a h.
\end{aligned}
\end{equation}
Thus, we can construct an L-system of the form
\begin{equation}\label{e-4-36-a}
    \Theta_a= \begin{pmatrix} T_a&K&\ J_a\cr  \calH^2 & &\dC^2\cr \end{pmatrix},
\end{equation}
where operators $T_a$, $K$,  and $J_a$ are defined by \eqref{e-d-4-33},  \eqref{e4-23-d}, and \eqref{e-34-J}, respectively. Taking into account that
$$
(T_a-zI)^{-1}=\left[
          \begin{array}{cc}
            \frac{1}{\bar\lambda_0-zI} & 0 \\
            0 & -\frac{1}{\lambda_0+zI} \\
          \end{array}
        \right],
$$
we proceed with calculations of the transfer function $W_{\Theta_a}(z)$. We have
\begin{equation}\label{e-4-27-a}
    \begin{aligned}
    W_{\Theta_a} (z)&=I-2iK^\ast (T_a-zI)^{-1}K J_a=\left[
          \begin{array}{cc}
            1 & 0 \\
            0 & 1 \\
          \end{array}
        \right]-2i\IM \lambda_0\left[
          \begin{array}{cc}
            -\frac{1}{\bar\lambda_0-z} & 0 \\
            0 & \frac{1}{{\lambda_0}+z} \\
          \end{array}
        \right]\\
    &=\left[
          \begin{array}{cc}
            1 & 0 \\
            0 & 1 \\
          \end{array}
        \right]-\left[
          \begin{array}{cc}
            -\frac{\lambda_0-\bar\lambda_0}{\bar\lambda_0-z} & 0 \\
            0 & \frac{\lambda_0-\bar\lambda_0}{\lambda_0+z} \\
          \end{array}
        \right]=\left[
          \begin{array}{cc}
            \frac{\lambda_0-z}{\bar\lambda_0-z} & 0 \\
            0 & \frac{\bar\lambda_0+z}{\lambda_0+z} \\
          \end{array}
        \right].
    \end{aligned}
\end{equation}
{For the corresponding impedance function we have }
\begin{equation}\label{e-4-36-aa}
    V_{\Theta_a} (z)=K^\ast (\RE T_a - zI)^{-1} K=\left[
          \begin{array}{cc}
            \frac{\IM\lambda_0}{\RE\lambda_0-z} & 0 \\
            0 & -\frac{\IM\lambda_0}{\RE\lambda_0+z} \\
          \end{array}
        \right].
\end{equation}

\section{c-Entropy of dissipative and non-dissipative L-systems}\label{s5}

We begin with reminding the  definition of   L-system  c-Entropy introduced in \cite{BT-16}.
\begin{definition}\label{d-9-ent}
Let $\Theta$ be an L-system of the form \eqref{BL}. The quantity
\begin{equation}\label{e-80-entropy-def}
    \calS=-\tr \ln (|W_\Theta(-i)|),
\end{equation}
where $W_\Theta(z)$ is the transfer function of $\Theta$, is called the \textbf{coupling entropy} (or \textbf{c-Entropy}) of the L-system $\Theta$.
\end{definition}

 Note that if, in addition,  the point $z=i$ belongs to $\rho(T)$, then we also have that
\begin{equation}\label{e-80-entropy}
     \calS=\tr \ln (|W_\Theta(i)|).
\end{equation}

Our next  goal is to  calculate  the c-Entropies for the dissipative and non-dissipative cases  described in Sections \ref{s3} and \ref{s4}, respectively.

Let us start with the dissipative case.
\begin{theorem}\label{t-2}%
Let $\Theta_d$ be an L-system   of the form \eqref{e4-27-d} that is based upon  operator $T_d$ in \eqref{e4-22-d}. Then  the c-Entropy $\calS_d$ of $\Theta_d$ is
\begin{equation}\label{e-40-entropy}
     \calS_d=\ln{\frac{|\lambda_0|^2+2\IM\lambda_0+1}{|\lambda_0|^2-2\IM\lambda_0+1}}.
\end{equation}
\end{theorem}
\begin{proof}
Our plan is to find $\calS_d$ by the definition using \eqref{e-80-entropy-def}. Keeping in mind \eqref{e-4-35} we observe
\begin{equation}\label{e-36-ln}
\begin{aligned}
\ln\left|\frac{\bar\lambda_0+i}{\lambda_0+i}\right|
&=\ln\left|\frac{\RE\lambda_0+i(1-\IM\lambda_0)}{\RE\lambda_0+i(1+\IM\lambda_0)}\right|
=\ln\sqrt{\frac{(\RE\lambda_0)^2+(1-\IM\lambda_0)^2}{(\RE\lambda_0)^2+(1+\IM\lambda_0)^2}}\\
&=\frac{1}{2}\ln{\frac{|\lambda_0|^2-2\IM\lambda_0+1}{|\lambda_0|^2+2\IM\lambda_0+1}}.
\end{aligned}
\end{equation}

Similarly,
\begin{equation}\label{e-37-ln}
    \begin{aligned}
\ln\left|\frac{\lambda_0-i}{\bar\lambda_0-i}\right|
&=\ln\left|\frac{\RE\lambda_0+i(\IM\lambda_0-1)}{\RE\lambda_0-i(\IM\lambda_0+1)}\right|
=\ln\sqrt{\frac{(\RE\lambda_0)^2+(\IM\lambda_0-1)^2}{(\RE\lambda_0)^2+(\IM\lambda_0+1)^2}}\\
&=\frac{1}{2}\ln{\frac{|\lambda_0|^2-2\IM\lambda_0+1}{|\lambda_0|^2+2\IM\lambda_0+1}}.
    \end{aligned}
\end{equation}

Consequently,
\begin{equation}\label{e-41}
    \begin{aligned}
\calS_d&=-\tr\ln (|W_{\Theta_d}(-i)|)=-\tr\ln \left[
          \begin{array}{cc}
            \left|\frac{\bar\lambda_0+i}{\lambda_0+i}\right| & 0 \\
            0 & \left|\frac{\lambda_0-i}{{\bar\lambda_0}-i}\right| \\
          \end{array}
        \right]\\
        &=-\tr\left[
          \begin{array}{cc}
            \ln\left| \frac{\bar\lambda_0+i}{\lambda_0+i}\right| & 0 \\
            0 & \ln \left|\frac{\lambda_0-i}{{\bar\lambda_0}-i}\right| \\
          \end{array}
        \right]
\\
&=-\tr\left[
          \begin{array}{cc}
            \frac{1}{2}\ln{\frac{|\lambda_0|^2-2\IM\lambda_0+1}{|\lambda_0|^2+2\IM\lambda_0+1}} & 0 \\
            0 & \frac{1}{2}\ln{\frac{|\lambda_0|^2-2\IM\lambda_0+1}{|\lambda_0|^2+2\IM\lambda_0+1}} \\
          \end{array}
        \right]\\
&=-\ln{\frac{|\lambda_0|^2-2\IM\lambda_0+1}{|\lambda_0|^2+2\IM\lambda_0+1}}
=\ln{\frac{|\lambda_0|^2+2\IM\lambda_0+1}{|\lambda_0|^2-2\IM\lambda_0+1}}.
    \end{aligned}
\end{equation}
Hence, we have shown \eqref{e-40-entropy}.
\end{proof}

Now we are going to evaluate c-Entropy in the first non-dissipative case.
\begin{theorem}\label{t-9}%
Let $\Theta_m$ be an L-system   of the form \eqref{e-4-34} that is based upon  operator $T_m$ in\eqref{e-d-4-30}. Then  the c-Entropy $\calS_m$ of $\Theta_m$ is zero.
\end{theorem}
\begin{proof}
Once again we are going to find $\calS_m$ by the definition using \eqref{e-80-entropy-def}. Taking into account \eqref{e-4-35} and \eqref{e-80-entropy-def} we have
\begin{equation}\label{e-41-m}
    \begin{aligned}
\calS_m&=-\tr\ln (|W_{\Theta_m}(-i)|)=-\tr\ln \left[
          \begin{array}{cc}
            \left|\frac{\bar\lambda_0+i}{\lambda_0+i}\right| & 0 \\
            0 & \left|\frac{\lambda_0+i}{{\bar\lambda_0}+i}\right| \\
          \end{array}
        \right]\\
        &=-\tr\left[
          \begin{array}{cc}
            \ln\left| \frac{\bar\lambda_0+i}{\lambda_0+i}\right| & 0 \\
            0 & \ln \left|\frac{\lambda_0+i}{{\bar\lambda_0}+i}\right| \\
          \end{array}
        \right]=-\left(\ln\left| \frac{\bar\lambda_0+i}{\lambda_0+i}\right|+\ln \left|\frac{\lambda_0+i}{{\bar\lambda_0}+i}\right| \right)
\\
&=-\ln 1 =0.
    \end{aligned}
\end{equation}
Thus, $\calS_m=0$.
\end{proof}

At this point we are ready  to discuss the accumulative case.
\begin{theorem}\label{t-4}%
Let $\Theta_a$ be an L-system   of the form \eqref{e-4-36-a} that is based upon  operator $T_a$ in \eqref{e-d-4-33}. Then  the c-Entropy $\calS_a$ of $\Theta_a$ is
\begin{equation}\label{e-46-entropy}
     \calS_a=\ln{\frac{|\lambda_0|^2-2\IM\lambda_0+1}{|\lambda_0|^2+2\IM\lambda_0+1}}.
\end{equation}
\end{theorem}
\begin{proof}
Our plan is to find $\calS_a$ by the definition using \eqref{e-80-entropy-def}. Keeping in mind \eqref{e4-28-d} we reciprocate the fraction in \eqref{e-36-ln} to get
\begin{equation}\label{e-42-ln}
\ln\left|\frac{\lambda_0+i}{\bar\lambda_0+i}\right|
=\frac{1}{2}\ln{\frac{|\lambda_0|^2+2\IM\lambda_0+1}{|\lambda_0|^2-2\IM\lambda_0+1}}.
\end{equation}
Similarly, reciprocating the corresponding fraction in \eqref{e-37-ln} yields
\begin{equation}\label{e-48-ln}
    \ln\left|\frac{\bar\lambda_0-i}{\lambda_0-i}\right|
    =\frac{1}{2}\ln{\frac{|\lambda_0|^2+2\IM\lambda_0+1}{|\lambda_0|^2-2\IM\lambda_0+1}}.
\end{equation}

Consequently,
\begin{equation}\label{e-41-a}
    \begin{aligned}
\calS_a&=-\tr\ln (|W_{\Theta_a}(-i)|)=-\tr\ln \left[
          \begin{array}{cc}
            \left|\frac{\lambda_0+i}{\bar\lambda_0+i}\right| & 0 \\
            0 & \left|\frac{\bar\lambda_0-i}{{\lambda_0}-i}\right| \\
          \end{array}
        \right]\\
&=-\tr\left[
          \begin{array}{cc}
            \frac{1}{2}\ln{\frac{|\lambda_0|^2+2\IM\lambda_0+1}{|\lambda_0|^2-2\IM\lambda_0+1}} & 0 \\
            0 & \frac{1}{2}\ln{\frac{|\lambda_0|^2+2\IM\lambda_0+1}{|\lambda_0|^2-2\IM\lambda_0+1}} \\
          \end{array}
        \right]\\
&=-\ln{\frac{|\lambda_0|^2+2\IM\lambda_0+1}{|\lambda_0|^2-2\IM\lambda_0+1}}
=\ln{\frac{|\lambda_0|^2-2\IM\lambda_0+1}{|\lambda_0|^2+2\IM\lambda_0+1}}.
    \end{aligned}
\end{equation}
Hence, we have shown \eqref{e-46-entropy}.
\end{proof}
Note that since $\IM\lambda_0>0$ the c-Entropy in \eqref{e-46-entropy} is clearly negative.

\begin{figure}
  \begin{center}
\includegraphics[width=80mm]{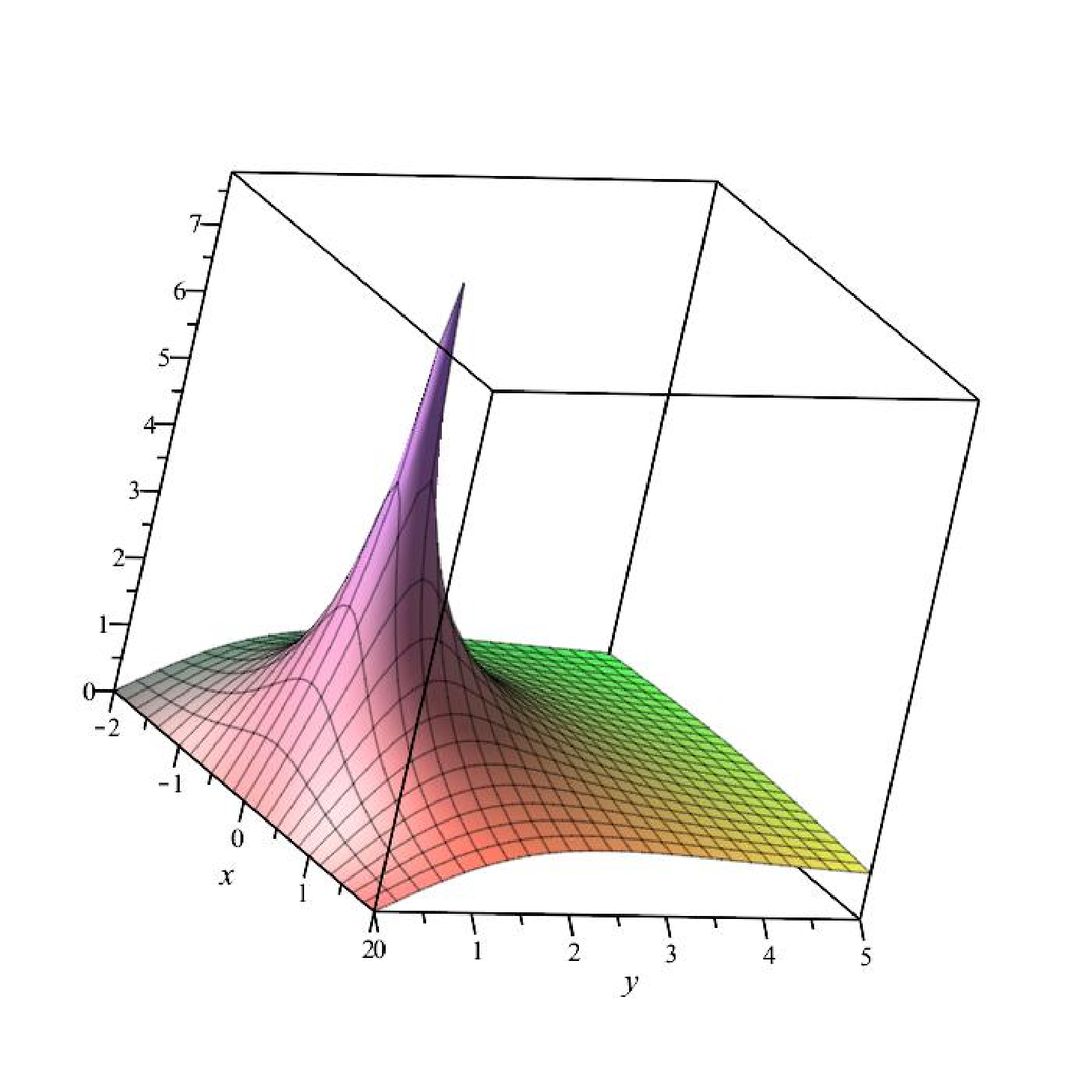}
  \caption{Maximal c-Entropy of a dissipative L-system}\label{fig-1}
  \end{center}
\end{figure}

We are going to analyze the L-systems $\Theta_d$ and $\Theta_a$ closer to see what values of $\lambda_0$ provide us with extreme values of c-Entropy.

\begin{theorem}\label{t-5}%
Let $\Theta_d$ be an L-system   of the form \eqref{e4-27-d} that is based upon  operator $T_d$ in \eqref{e4-22-d}. Then  the c-Entropy $\calS_d=+\infty$ if and only if $\lambda_0=i$. Moreover, the c-Entropy $\calS_d$ attains its  maximum (finite) value whenever  $\lambda_0=ai$, for $a\in(0,1)\cup(1,+\infty)$.
\end{theorem}
\begin{proof}
Applying formula \eqref{e-40-entropy} for the c-Entropy $\calS_d$ and taking into account that $\IM\lambda_0>0$, we observe that the numerator of the expression
\begin{equation}\label{e-40-fr}
   \frac{|\lambda_0|^2+2\IM\lambda_0+1}{|\lambda_0|^2-2\IM\lambda_0+1}
\end{equation}
is always greater than $1$. Consequently, the natural logarithm expression on the right side of \eqref{e-40-entropy} is positive infinity if and only if the denominator in \eqref{e-40-fr} is zero, or
$$
|\lambda_0|^2-2\IM\lambda_0+1=(\RE\lambda_0)^2+(1-\IM\lambda_0)^2=0.
$$
The last equation is true whenever $\lambda_0=i$.

The second part of the statement is proved by considering the expression
$$
\frac{|\lambda_0|^2+2\IM\lambda_0+1}{|\lambda_0|^2-2\IM\lambda_0+1}=\frac{(\RE\lambda_0)^2+(1+\IM\lambda_0)^2}{(\RE\lambda_0)^2+(1-\IM\lambda_0)^2}
$$
as a function of a single real variable $x=\RE\lambda_0$ while keeping $\IM\lambda_0=a>0$, $a\ne1$ fixed. Then
$$
f(x)=\frac{x^2+(1+a)^2}{x^2+(1-a)^2}
$$
clearly attains its maximum at $x=0$.
\end{proof}

\begin{figure}
  \begin{center}
  \includegraphics[width=80mm]{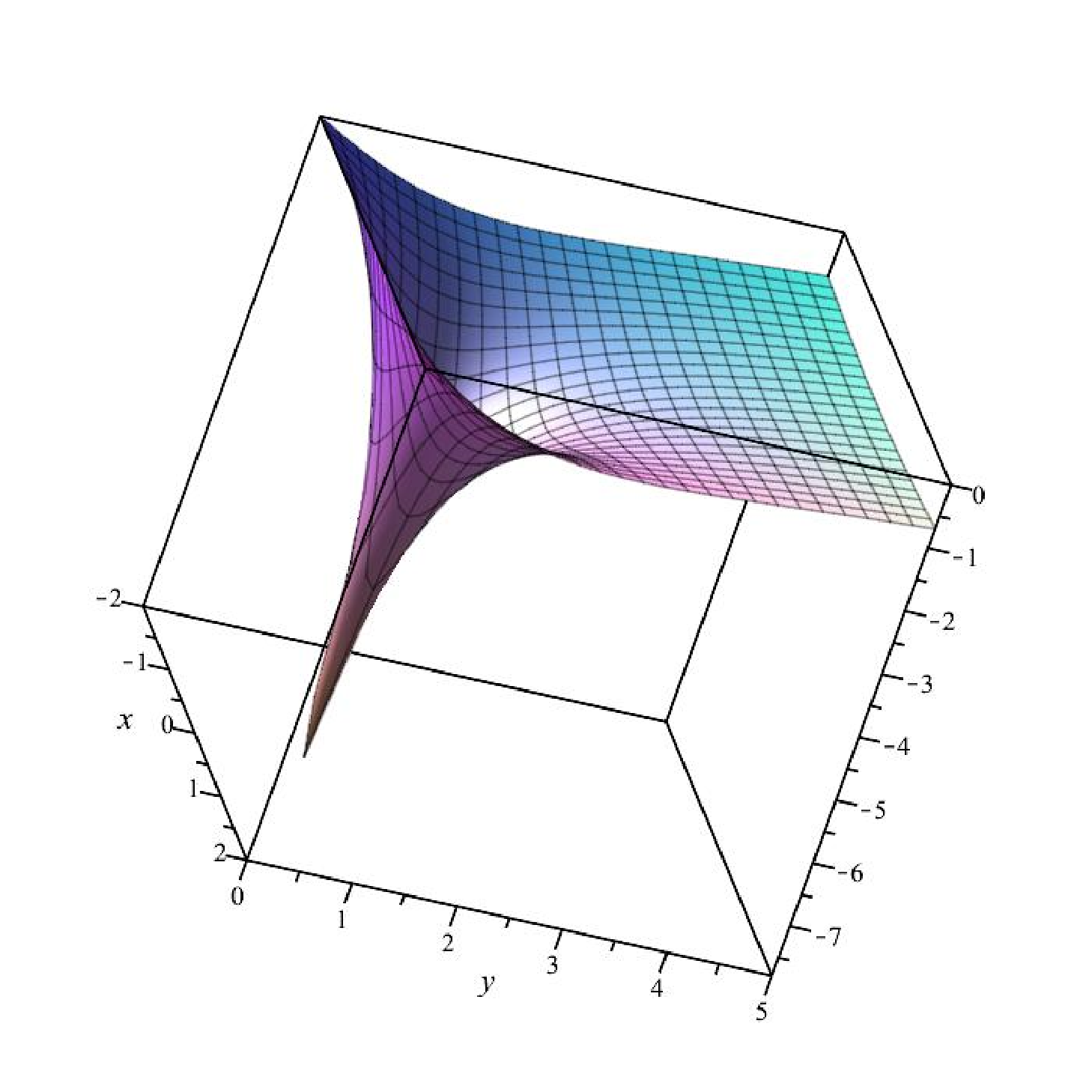}
  \caption{Minimal c-Entropy of an accumulative  L-system}\label{fig-2}
  \end{center}
\end{figure}

 The graph of c-Entropy {$\calS_d(x,y)$ } as a function of $x=\RE\lambda_0$ and $y=\RE\lambda_0$ is shown on Figure \ref{fig-1}. The visible pick is actually represents the infinite value of {}{$\calS_d(x,y)$ } that occurs at the point $(0,1)$ when $\lambda_0=i$.

A somewhat similar result takes place when we consider an accumulative L-system $\Theta_a$.
\begin{theorem}\label{t-6}%
Let $\Theta_a$ be an L-system   of the form \eqref{e-4-36-a} that is based upon  operator $T_a$ in \eqref{e-d-4-33}. Then  the c-Entropy $\calS_a=-\infty$ if and only if $\lambda_0=i$. Moreover, the c-entropy $\calS_d$ attains its  minimum (finite) value whenever  $\lambda_0=ai$, for $a\in(0,1)\cup(1,+\infty)$.
\end{theorem}
\begin{proof}
The proof relies on the formula \eqref{e-46-entropy} and  closely replicates the steps of the proof of Theorem \ref{t-5} but applied to the expression
\begin{equation}\label{e-41-fr}
   \frac{|\lambda_0|^2-2\IM\lambda_0+1}{|\lambda_0|^2+2\IM\lambda_0+1}
   =\frac{(\RE\lambda_0)^2+(1-\IM\lambda_0)^2}{(\RE\lambda_0)^2+(1+\IM\lambda_0)^2}.
\end{equation}
Here we observe that the natural logarithm expression on the right side of \eqref{e-46-entropy} is negative infinity if and only if the numerator in \eqref{e-41-fr} is zero, or (as we have shown in the proof of Theorem \ref{t-5})
$$
|\lambda_0|^2-2\IM\lambda_0+1=(\RE\lambda_0)^2+(1-\IM\lambda_0)^2=0.
$$
The last equation is true whenever $\lambda_0=i$.

The second part of the statement is proved by analyzing the function
$$
g(x)=\frac{x^2+(1-a)^2}{x^2+(1+a)^2}
$$
to show that it achieves its minimum at $x=0$.
\end{proof}
 The graph of c-Entropy {$\calS_a(x,y)$ } as a function of $x=\RE\lambda_0$ and $y=\RE\lambda_0$ is shown on Figure \ref{fig-2}. The visible drop actually represents the infinite value of {$\calS_a(x,y)$ } that occurs at the point $(0,1)$ when $\lambda_0=i$.

\section{Coefficients of dissipation and accumulation}\label{s6}

Let us recall the definition of the dissipation coefficient of a L-system with the main dissipative operator $T$.
\begin{definition}[{cf. \cite{BT-16}}, \cite{BT-21}]\label{d-10}
Let $\Theta$ be an L-system  of the form \eqref{BL} with the main  operator $T$ and c-Entropy $\calS$. Then the quantity
 \begin{equation}\label{e-69-ent-dis}
\calD=1-e^{-2\calS}
\end{equation}
 is called the {\textbf{coefficient of dissipation}} of the L-system $\Theta$ if $\calS\ge0$.

Similarly,  the quantity
   \begin{equation}\label{e-70-ent-acc}
\calA=1-e^{2\calS}
\end{equation}
is called the  {\textbf{coefficient of accumulation}} of the L-system $\Theta$ if $\calS<0$.
\end{definition}

We will evaluate the coefficient of dissipation (or accumulation) for all three cases described in Section \ref{s6}. We begin with the case of an L-system   of the form \eqref{e4-27-d} with the dissipative operator $T_d$.

\begin{theorem}\label{t-8}%
Let $\Theta_d$ be an L-system   of the form \eqref{e4-27-d} that is based upon  operator $T_d$ in \eqref{e4-22-d}. Then  the dissipation coefficient $\calD_d$ of $\Theta_d$ is
\begin{equation}\label{e-44-dc}
     \calD_d=\frac{8\IM\lambda_0(|\lambda_0|^2+1)}{(|\lambda_0|^2+2\IM\lambda_0+1)^2}.
\end{equation}
\end{theorem}
\begin{proof}
In order to simplify calculations let us temporarily denote $\lambda_0=a+bi$. Using \eqref{e-69-ent-dis} with   \eqref{e-40-entropy} and performing straightforward calculations we get
$$
\begin{aligned}
\calD_d&=1-e^{-2\cS_d}=1-\left(\frac{|\lambda_0|^2-2\IM\lambda_0+1}{|\lambda_0|^2+2\IM\lambda_0+1}\right)^2
=1-\frac{\big(a^2+(1-b)^2\big)^2}{\big(a^2+(1+b)^2\big)^2}\\
&=\frac{\big(a^2+(1+b)^2\big)^2-\big(a^2+(1-b)^2\big)^2}{\big(a^2+(1+b)^2\big)^2}=\frac{8a^2b+8b+8b^3}{\big(a^2+(1+b)^2\big)^2}\\
&=\frac{8b(a^2+b^2+1)}{\big(a^2+(1+b)^2\big)^2}=\frac{8\IM\lambda_0(|\lambda_0|^2+1)}{(|\lambda_0|^2+2\IM\lambda_0+1)^2}.
\end{aligned}
$$
Thus, \eqref{e-44-dc} takes place.
\end{proof}

Nest we will treat the case when the main operator of our L-system is neither dissipative nor accumulative.

\begin{theorem}\label{t-7}%
Let $\Theta_m$ be an L-system   of the form \eqref{e-4-34} that is based upon  operator $T_m$ in \eqref{e-d-4-30}. Then  the dissipation coefficient $\calD_m$ of $\Theta_m$ is zero, that is $\calD_m=0$.
\end{theorem}
\begin{proof}
As we have shown in Theorem \ref{t-9}, for an L-system   of the form \eqref{e-4-34} that is based upon  operator $T_m$ in \eqref{e-d-4-30} the value of c-Entropy is zero. Hence, \eqref{e-69-ent-dis} yields
$$
\calD_m=1-e^{-2\cdot 0}=1-1=0.
$$
\end{proof}

Finally we are going to look at the case when the main operator of our L-system is  accumulative.
\begin{theorem}\label{t-10}%
Let $\Theta_a$ be an L-system   of the form \eqref{e-4-36-a} that is based upon  operator $T_a$ in \eqref{e-d-4-33}. Then  the accumulation coefficient $\calA_a$ of $\Theta_a$ is
\begin{equation}\label{e-52-dc-a}
     \calA_a=\frac{8\IM\lambda_0(|\lambda_0|^2+1)}{(|\lambda_0|^2+2\IM\lambda_0+1)^2}.
\end{equation}
\end{theorem}
\begin{proof}
As in the proof of Theorem \ref{t-8} we let temporarily  $\lambda_0=a+bi$. Using \eqref{e-70-ent-acc} with   \eqref{e-46-entropy} and performing straightforward calculations we obtain
$$
\begin{aligned}
\calA_a&=1-e^{2\cS_a}=1-\left(\frac{|\lambda_0|^2-2\IM\lambda_0+1}{|\lambda_0|^2+2\IM\lambda_0+1}\right)^2
=1-\frac{\big(a^2+(1-b)^2\big)^2}{\big(a^2+(1+b)^2\big)^2}\\
&=\frac{8b(a^2+b^2+1)}{\big(a^2+(1+b)^2\big)^2}=\frac{8\IM\lambda_0(|\lambda_0|^2+1)}{(|\lambda_0|^2+2\IM\lambda_0+1)^2}.
\end{aligned}
$$
Thus, \eqref{e-52-dc-a} takes place.
\end{proof}

\section{L-system coupling and c-Entropy}\label{s7}

In this section, following \cite{Bro} (see also \cite{ABT}, \cite{BMkT-2} and \cite{BT-23}), we  introduce the coupling of two L-systems based upon the multiplication operator discussed in Sections \ref{s3} and \ref{s4}.

Let $\calH^2$ be a two-dimensional Hilbert space with an inner product $(\cdot,\cdot)$ and an orthogonal normalized basis of vectors $h_{01},h_{02}\in \calH^2$, ($\|h_{01}\|=\|h_{02}\|=1$). For fixed numbers $\lambda_1,\,\lambda_2,\,\mu_1,\mu_2\in\dC$ and such that $\IM\lambda_1>0$ and $\IM\lambda_2>0$, we introduce linear operators $T_1$ and $T_2$  in $\calH^2$ of the form
\begin{equation}\label{e-45-12}
    T_1 h=\left[
          \begin{array}{cc}
            \lambda_1 & 0 \\
            0 & \mu_1 \\
          \end{array}
        \right]
     h,\quad  T_2 h=\left[
          \begin{array}{cc}
            \lambda_2 & 0 \\
            0 & \mu_2 \\
          \end{array}
        \right]
     h,\quad h=\left[
                               \begin{array}{c}
                                 h_1 \\
                                 h_2 \\
                               \end{array}
                             \right]   \in\calH^2.
\end{equation}
Clearly, for $j=1,2$
$$
T^*_{j}=\left[
          \begin{array}{cc}
            \bar\lambda_j & 0 \\
            0 & \bar\mu_j \\
          \end{array}
        \right],\;\IM T_j =\left[
          \begin{array}{cc}
            \IM\lambda_j & 0 \\
            0 & \IM\mu_j \\
          \end{array}
        \right],\; \RE T_j =\left[
          \begin{array}{cc}
            \RE\lambda_j & 0 \\
            0 & \RE\mu_j \\
          \end{array}
        \right].
$$
Let $K_j:\dC^2\rightarrow\calH^2$ (for $j=1,2$) be such that
\begin{equation}\label{e4-23-lm}
   K_j \left[
                               \begin{array}{c}
                                 c_1 \\
                                 c_2 \\
                               \end{array}
                             \right]=\left[
          \begin{array}{cc}
            (\sqrt{\IM\lambda_j})h_{01} & 0 \\
            0 & (\sqrt{\IM\mu_j})h_{02} \\
          \end{array}
        \right]\left[
                               \begin{array}{c}
                                 c_1 \\
                                 c_2 \\
                               \end{array}
                             \right]=\left[
                               \begin{array}{c}
                               (\sqrt{\IM\lambda_j})c_1   h_{01} \\
                                (\sqrt{\IM\mu_j}) c_2  h_{02}\\
                               \end{array}
                             \right],
\end{equation}
for all $c_1,\,c_2\in\dC$. Then, the adjoint operators $K^*_j:\calH^2\rightarrow\dC^2$ (for $j=1,2$) act on an arbitrary vector
$h=\left[
                               \begin{array}{c}
                                 h_1 \\
                                 h_2 \\
                               \end{array}
                             \right] =\left[
                               \begin{array}{c}
                               c_1   h_{01} \\
                                 c_2  h_{02}\\
                               \end{array}
                             \right]\in\calH^2$ as follows
\begin{equation}\label{e-4-32-lm}
   K^*_j h=K^*_j\left[
                               \begin{array}{c}
                               c_1   h_{01} \\
                                 c_2  h_{02}\\
                               \end{array}
                             \right]=\left[
                               \begin{array}{c}
                               (\sqrt{\IM\lambda_j})(h,c_1   h_{01}) \\
                                 (\sqrt{\IM\mu_j})(h, c_2  h_{02})\\
                               \end{array}
                             \right]=\left[
                               \begin{array}{c}
                                (\sqrt{\IM\lambda_j}) c_1 \\
                                  (\sqrt{\IM\mu_j})c_2 \\
                               \end{array}
                             \right].
\end{equation}

Let also
\begin{equation}\label{e48-J}
    J=I=\left[
          \begin{array}{cc}
            1 & 0 \\
            0 & 1 \\
          \end{array}
        \right]
\end{equation}
be the identity operator. Furthermore,
\begin{equation}\label{e49-lm}
\begin{aligned}
K_jJK^*_j h&=\left[
          \begin{array}{cc}
            (\sqrt{\IM\lambda_j})h_{01} & 0 \\
            0 & (\sqrt{\IM\mu_j})h_{02} \\
          \end{array}
        \right]\left[
          \begin{array}{cc}
            1 & 0 \\
            0 & 1 \\
          \end{array}
        \right]\left[
                               \begin{array}{c}
                                (\sqrt{\IM\lambda_j}) c_1 \\
                                  (\sqrt{\IM\mu_j})c_2 \\
                               \end{array}
                             \right]\\
                             &=\left[
          \begin{array}{cc}
            \IM\lambda_j & 0 \\
            0 & \IM\mu_j \\
          \end{array}
        \right]\left[
                               \begin{array}{c}
                                 h_1 \\
                                 h_2 \\
                               \end{array}
                             \right] =\IM T_j h.
\end{aligned}
\end{equation}

Consider two L-systems  based on $T_1$ and $T_2$, respectively,
\begin{equation}\label{e-46-Theta12}
    \Theta_1= \begin{pmatrix} T_1&K_1&\ J\cr  \calH^2 & &\dC^2\cr \end{pmatrix}
\quad \textrm{and}\quad
\Theta_2= \begin{pmatrix} T_2&K_2&\ J\cr  \calH^2 & &\dC^2\cr \end{pmatrix}.
\end{equation}
Taking into account that
\begin{equation}\label{e-51-Rezs}
(T_j-zI)^{-1}=\left[
          \begin{array}{cc}
            \frac{1}{\lambda_j-zI} & 0 \\
            0 & \frac{1}{\mu_j-zI} \\
          \end{array}
        \right],\quad j=1,2,
\end{equation}
we proceed with calculations of the transfer function $W_{\Theta_j}(z)$. Performing steps similar to the ones that were used to obtain formula \eqref{e-4-35}, we get
\begin{equation}\label{e-4-51}
        W_{\Theta_j} (z)=I-2iK^\ast_j (T_j-zI)^{-1}K_jJ=\left[
          \begin{array}{cc}
            \frac{\bar\lambda_j-z}{\lambda_j-z} & 0 \\
            0 & \frac{\mu_j-z}{\overline{\mu_j}-z} \\
          \end{array}
        \right],\quad j=1,2.
\end{equation}

Define an operator $\mathbf{T}$ acting on the Hilbert space $\calH^2\oplus\calH^2$ as
\begin{equation}\label{e-48-bT}
  \begin{aligned}
   \mathbf{T}\left[
               \begin{array}{c}
                 \mathbf{h}_1 \\
                 \mathbf{h}_2 \\
               \end{array}
             \right]&=\left[
                       \begin{array}{cc}
                         T_1 & 2iK_1JK_2^* \\
                         0 & T_2 \\
                       \end{array}
                     \right]\left[ \begin{array}{c}
                 \mathbf{h}_1 \\
                 \mathbf{h}_2 \\
               \end{array}
             \right]=\left[
               \begin{array}{c}
                 T_1 \mathbf{h}_1+2i K_1K_2^*\mathbf{h}_2 \\
                 T_2 \mathbf{h}_2 \\
               \end{array}
             \right]\\
             &\\
             &=\left[
                 \begin{array}{cccc}
                   \lambda_1 & 0 & 2i\sqrt{\IM\lambda_1\IM\lambda_2} & 0 \\
                   0 & \mu_1 & 0 & 2i\sqrt{\IM\mu_1\IM\mu_2} \\
                   0 & 0 & \lambda_2 & 0 \\
                   0 & 0 & 0 & \mu_2 \\
                 \end{array}
               \right]\left[
                        \begin{array}{c}
                          h_{11} \\
                          h_{12} \\
                          h_{21} \\
                          h_{22} \\
                        \end{array}
                      \right].
      \end{aligned}
   \end{equation}
Here
$$\mathbf{h}_1=\left[ \begin{array}{c}
                 h_{11} \\
                 h_{12} \\
               \end{array}
               \right],\mathbf{h}_2=\left[ \begin{array}{c}
                 h_{21} \\
                 h_{22} \\
               \end{array}
               \right]\in\calH^2 .$$

Direct check reveals that
 \begin{equation}\label{e-50-bT-adj}
     \mathbf{T^*}\left[
               \begin{array}{c}
                 \mathbf{h}_1 \\
                 \mathbf{h}_2 \\
               \end{array}
             \right]=\left[
                 \begin{array}{cccc}
                   \bar\lambda_1 & 0 & 0 & 0 \\
                   0 & \bar\mu_1 & 0 & 0 \\
                   -2i\sqrt{\IM\lambda_1\IM\lambda_2} & 0 & \bar\lambda_2 & 0 \\
                   0 & -2i\sqrt{\IM\mu_1\IM\mu_2} & 0 & \bar\mu_2 \\
                 \end{array}
               \right]\left[
                        \begin{array}{c}
                          h_{11} \\
                          h_{12} \\
                          h_{21} \\
                          h_{22} \\
                        \end{array}
                      \right].
        \end{equation}
In addition to $\mathbf{T}$ we define an operator $\mathbf{K}:\dC^2\rightarrow\calH^2\oplus\calH^2$ as
\begin{equation}\label{e-50-bK}
   \mathbf{K\,c}= \mathbf{K}\left[
               \begin{array}{c}
                 c_1 \\
                 c_2 \\
               \end{array}
             \right]=\left[
               \begin{array}{c}
                 K_1 \mathbf{c} \\
                 K_2\mathbf{c}\\
               \end{array}
             \right]=\left[
               \begin{array}{c}
                 (\sqrt{\IM\lambda_1}\, c_1) h_{01} \\
                (\sqrt{\IM\mu_1}\, c_2) h_{02}\\
                (\sqrt{\IM\lambda_2}\, c_1) h_{01} \\
                (\sqrt{\IM\mu_2}\, c_2) h_{02}\\
               \end{array}
             \right],\;\mathbf{c}=\left[
               \begin{array}{c}
                 c_1 \\
                 c_2 \\
               \end{array}
             \right].
\end{equation}
In this case the adjoint operator $\mathbf{K}^*:\calH^2\oplus\calH^2\rightarrow\dC^2$ is defined for all $\mathbf{h}_1,\mathbf{h}_2\in\calH^2$ as follows
\begin{equation}\label{e-51-b-adj}
    \mathbf{K}^*\left[
               \begin{array}{c}
                 h_1 \\
                 h_2 \\
               \end{array}
             \right]=K_1^*\mathbf{h}_1+K_2^*\mathbf{h}_2
             =\left[
               \begin{array}{c}
                 (h_{11}, \sqrt{\IM\lambda_1}\,h_{01})+(h_{12}, \sqrt{\IM\mu_1}\,h_{02}) \\
                 (h_{21}, \sqrt{\IM\lambda_2}\,h_{01})+(h_{22}, \sqrt{\IM\mu_2}\,h_{02}) \\
               \end{array}
             \right].
\end{equation}
One can confirm that
\begin{equation}\label{e-53-ImbT}
     \begin{aligned}
    \IM \mathbf{T}\left[
               \begin{array}{c}
              \mathbf{h}_1     \\
               \mathbf{h}_2     \\
               \end{array}
             \right]&=
             \frac{1}{2i}\left(
                       \begin{array}{cc}
                         2i\IM T_1 & 2iK_1K_2^* \\
                         2iK_2K_1^* & 2i\IM T_2 \\
                       \end{array}
                     \right)
             \left[
               \begin{array}{c}
              \mathbf{h}_1     \\
              \mathbf{h}_2     \\
               \end{array}
             \right]\\
             &=\left[
               \begin{array}{c}
                K_1K_1^* \mathbf{h}_1+K_1K_2^* \mathbf{h}_2 \\
                K_2K_2^* \mathbf{h}_2+K_2K_1^* \mathbf{h}_1\\
               \end{array}
             \right]=\mathbf{K}\mathbf{K}^*\left[
               \begin{array}{c}
                 \mathbf{h}_1 \\
                 \mathbf{h}_2 \\
               \end{array}
             \right],
        \end{aligned}
\end{equation}
and hence $\IM \mathbf{T}=\mathbf{K}\mathbf{K}^*$.

{Summarizing, we arrive at the following definition.}
\begin{definition}\label{d-11}
Given two systems of the form \eqref{e-46-Theta12}
$$
  \Theta_1= \begin{pmatrix} T_1&K_1&\ J\cr  \calH^2 & &\dC^2\cr \end{pmatrix}
\quad \textrm{and}\quad
\Theta_2= \begin{pmatrix} T_2&K_2&\ J\cr  \calH^2 & &\dC^2\cr \end{pmatrix},
$$
define the  \textbf{coupling of two L-systems} as
 \begin{equation*}
\Theta= \begin{pmatrix} \mathbf{T}&\mathbf{K}&\ J\cr \calH^2\oplus\calH^2& &\dC^2\cr
\end{pmatrix},
\end{equation*}
where the operators $\mathbf{T}$ and $\mathbf{K}$  are presented in \eqref{e-48-bT}--\eqref{e-51-b-adj}.

In writing,
$$
\Theta=\Theta_1\cdot\Theta_2.
$$
\end{definition}

The following theorem can be derived from a more general result \cite[Theorem 3.1]{Bro},
{however, for {the} convenience of the reader, we provide a simple  proof of it.}
 \begin{theorem}[cf. \cite{Bro}]\label{unitar}
 Let an L-system $\Theta$ be the coupling of two L-systems $\Theta_1$
and $\Theta_2$ of the form \eqref{e-46-Theta12} with the main operators $T_1$ and $T_2$ given by \eqref{e-45-12}. Then if $z\in\rho(T_1)\cap\rho(T_2)={\bbC\setminus\{\lambda_1,\lambda_2,\mu_1,\mu_2\}}$, we have
\begin{equation}\label{e-91-mult}
W_\Theta(z)=W_{\Theta_1}(z)\cdot W_{\Theta_2}(z)=\left[
          \begin{array}{cc}
            \frac{\bar\lambda_1-z}{\lambda_1-z}\cdot\frac{\bar\lambda_2-z}{\lambda_2-z} & 0 \\
            0 & \frac{\bar\mu_1-z}{{\mu_1}-z}\cdot\frac{\bar\mu_2-z}{{\mu_2}-z} \\
          \end{array}
        \right].
\end{equation}
\end{theorem}
\begin{proof}
We start with a simple observation that for $\mathbf{T}$ defined by \eqref{e-48-bT} we have
\begin{equation}\label{e-58-T-z}
    \mathbf{T}-zI=\left[
                       \begin{array}{cc}
                         T_1-zI & 2iK_1JK_2^* \\
                         0 & T_2-zI \\
                       \end{array}
                     \right]
\end{equation}
Then a direct check would confirm that
\begin{equation}\label{e-59-ResT}
    (\mathbf{T}-zI)^{-1}=\left[
                       \begin{array}{cc}
                         (T_1-zI)^{-1} & -2i(T_1-zI)^{-1}K_1JK_2^*(T_2-zI)^{-1} \\
                         0 & (T_2-zI)^{-1} \\
                       \end{array}
                     \right].
\end{equation}
Applying formulas \eqref{e-51-Rezs}, \eqref{e4-23-lm}--\eqref{e48-J} together with \eqref{e-59-ResT} one gets
$$
(\mathbf{T}-zI)^{-1}=\left[
                 \begin{array}{cccc}
                   \frac{1}{\lambda_1-zI} & 0 & \frac{-2i\sqrt{\IM\lambda_1\IM\lambda_2}}{(\lambda_1-zI)(\lambda_2-zI)} & 0 \\
                   0 & \frac{1}{\mu_1-zI} & 0 & \frac{-2i\sqrt{\IM\mu_1\IM\mu_2}}{(\mu_1-zI)(\mu_2-zI)} \\
                   0 & 0 & \frac{1}{\lambda_2-zI} & 0 \\
                   0 & 0 & 0 & \frac{1}{\mu_2-zI} \\
                 \end{array}
               \right].
$$
Taking in to account that
$$
2i\sqrt{\IM\lambda_1\cdot\IM\lambda_2}=\sqrt{(\lambda_1-\bar\lambda_1)(\lambda_2-\bar\lambda_2)},
$$
 and also
$$
  \begin{aligned}
&\frac{2i\left((\lambda_2-z)\IM\lambda_1-2i\IM\lambda_1\IM\lambda_2+(\lambda_1-z)\IM \lambda_1 \right)}{(\lambda_1-z)(\lambda_2-z)}\\
&=\frac{(\lambda_1-\bar\lambda_1)(\bar\lambda_2-z)+(\lambda_1-z)(\lambda_2-\bar\lambda_2)}{(\lambda_1-z)(\lambda_2-z)},
\end{aligned}
$$
(the same is true for $\mu_1$ and $\mu_2$) we   proceed to find
$$
 2i\mathbf{K}^*(\mathbf{T}-zI)^{-1}\mathbf{K}=\left[
          \begin{array}{cc}
            \frac{(\lambda_1-\bar\lambda_1)(\bar\lambda_2-z)+(\lambda_1-z)(\lambda_2-\bar\lambda_2)}{(\lambda_1-z)(\lambda_2-z)} & 0 \\
            0 & \frac{(\mu_1-\bar\mu_1)(\bar\mu_2-z)+(\mu_1-z)(\mu_2-\bar\mu_2)}{(\mu_1-z)(\mu_2-z)} \\
          \end{array}
        \right].
 $$
 Next we observe that
 $$
 \begin{aligned}
 1-&\frac{(\lambda_1-\bar\lambda_1)(\bar\lambda_2-z)+(\lambda_1-z)(\lambda_2-\bar\lambda_2)}{(\lambda_1-z)(\lambda_2-z)}\\
 &=\frac{(\lambda_1-z)(\lambda_2-z)-(\lambda_1-\bar\lambda_1)(\bar\lambda_2-z)-(\lambda_1-z)(\lambda_2-\bar\lambda_2)}{(\lambda_1-z)(\lambda_2-z)} \\
 &=\frac{(\bar\lambda_1-z)(\bar\lambda_2-z)}{(\lambda_1-z)(\lambda_2-z)}.
  \end{aligned}
 $$
 Finally,
  $$
    \begin{aligned}
W_\Theta(z)&=I-2i\mathbf{K}^*(\mathbf{T}-zI)^{-1}\mathbf{K}\\
&=I-\left[
          \begin{array}{cc}
            \frac{(\lambda_1-\bar\lambda_1)(\bar\lambda_2-z)+(\lambda_1-z)(\lambda_2-\bar\lambda_2)}{(\lambda_1-z)(\lambda_2-z)} & 0 \\
            0 & \frac{(\mu_1-\bar\mu_1)(\bar\mu_2-z)+(\mu_1-z)(\mu_2-\bar\mu_2)}{(\mu_1-z)(\mu_2-z)} \\
          \end{array}
        \right]\\
&=\left[
          \begin{array}{cc}
            \frac{(\bar\lambda_1-z)(\bar\lambda_2-z)}{(\lambda_1-z)(\lambda_2-z)} & 0 \\
            0 & \frac{(\bar\mu_1-z)(\bar\mu_2-z)}{(\mu_1-z)(\mu_2-z)} \\
          \end{array}
        \right]=W_{\Theta_1}(z)\cdot W_{\Theta_2}(z),
 \end{aligned}
 $$
 completing the proof.
\end{proof}

Now, we turn our attention to the c-Entropy of L-system coupling. The following theorem establishes the additivity property of c-Entropy with respect to the coupling of two L-systems, thereby justifying the use of the term ``coupling entropy."

\begin{theorem}\label{t-13}
Let an L-system $\Theta$ be the coupling of two L-systems $\Theta_1$ and $\Theta_2$ of the form \eqref{e-46-Theta12} with the corresponding c-entropies $\calS_1$ and $\calS_2$.
Then the c-Entropy $\calS$ of $\Theta$ is such that
\begin{equation}\label{e-56-ent}
\calS=\calS_1+\calS_2.
\end{equation}

If either $\calS_1=\infty$ or $\calS_2=\infty$, then $\calS=\infty$.
\end{theorem}
\begin{proof}
We rely on the Definition \ref{d-9-ent} of c-Entropy and Theorem \ref{unitar}. Applying \eqref{e-80-entropy-def} with \eqref{e-91-mult} yields
$$
\begin{aligned}
\calS&=-\tr\ln (|W_\Theta(-i)|)=-\tr \ln (|W_{\Theta_1}(-i)\cdot W_{\Theta_2}(-i)|)\\
&=-\tr\ln (|W_{\Theta_1}(-i)|-\tr\ln (|W_{\Theta_2}(-i)|=\calS_1+\calS_2.
\end{aligned}
$$
\end{proof}
Now we are going to look into the dissipation and accumulation coefficients of the coupling of two L-systems of the form \eqref{e-46-Theta12}. In \cite{BT-21} we made a note that  if L-system $\Theta$ with c-Entropy $\calS$ is a coupling of two L-systems $\Theta_1$ and $\Theta_2$ with c-entropies $\calS_1$ and $\calS_2$, respectively, formula \eqref{e-56-ent} holds, i.e., $\calS=\calS_1+\calS_2.$
Let $\calD$, $\calD_1$, and $\calD_2$ be the dissipation coefficients of L-systems $\Theta$, $\Theta_1$, and $\Theta_2$. Then \eqref{e-69-ent-dis} implies
$$
1-\calD=e^{-2\cS}=e^{-2(\cS_1+\cS_2)}=e^{-2\cS_1}\cdot e^{-2\cS_2}=(1-\calD_1)(1-\calD_2).
$$
Thus, the formula
\begin{equation}\label{e-72-coupling}
    \calD=1-(1-\calD_1)(1-\calD_2)=\calD_1+\calD_2-\calD_1\calD_2
\end{equation}
describes the coefficient of dissipation of the L-system coupling. Similarly,
$$
1-\calA=e^{2\cS}=e^{2(\cS_1+\cS_2)}=e^{2\cS_1}\cdot e^{2\cS_2}=(1-\calA_1)(1-\calA_2).
$$
Therefore,
\begin{equation}\label{e-73-coupling}
    \calA=1-(1-\calA_1)(1-\calA_2)=\calA_1+\calA_2-\calA_1\calA_2
\end{equation}
is the coefficient of accumulation of the L-system coupling.

\section{Examples}

In this section we present  examples that illustrate the construction of L-system of the forms \eqref{e4-27-d}, \eqref{e-4-34}, and \eqref{e-4-36-a} for {various} values of $\lambda_0$.

\subsection*{Example 1}\label{ex-1}

Let $\lambda_0=i$ and consider the linear operators $T_d$, $T_m$, and $T_a$ of the forms \eqref{e4-22-d}, \eqref{e-d-4-30}, and \eqref{e-d-4-33}, respectively, acting on  a two-dimensional Hilbert space $\calH^2$ with an inner product $(\cdot,\cdot)$ and an orthogonal normalized basis of vectors $h_{01},h_{02}\in \calH^2$, ($\|h_{01}\|=\|h_{02}\|=1$). We have then
\begin{equation}\label{e-ex-63}
    T_d h=\left[
          \begin{array}{cc}
            i & 0 \\
            0 & i \\
          \end{array}
        \right]
     h,\quad T_m h=\left[
          \begin{array}{cc}
            i & 0 \\
            0 & -i \\
          \end{array}
        \right]
     h,\quad T_a h=\left[
          \begin{array}{cc}
            -i & 0 \\
            0 & -i \\
          \end{array}
        \right],
\end{equation}
where $h=\left[
                               \begin{array}{c}
                                 h_1 \\
                                 h_2 \\
                               \end{array}
                             \right]   \in\calH^2$.

We are going to include $T_d$, $T_m$, and $T_a$  into  L-systems $\Theta_d$, $\Theta_m$, and $\Theta_a$ of the forms \eqref{e4-27-d}, \eqref{e-4-34}, and \eqref{e-4-36-a}, respectively. In order to do that we take an operator $K:\dC^2\rightarrow\calH^2$ of the form \eqref{e4-23-d} that in our case of $\IM i=1$ becomes
\begin{equation}\label{e-ex1-73}
    K \left[
                               \begin{array}{c}
                                 c_1 \\
                                 c_2 \\
                               \end{array}
                             \right]=\left[
          \begin{array}{cc}
            h_{01} & 0 \\
            0 & h_{02} \\
          \end{array}
        \right]\left[
                               \begin{array}{c}
                                 c_1 \\
                                 c_2 \\
                               \end{array}
                             \right]=\left[
                               \begin{array}{c}
                               c_1   h_{01} \\
                                 c_2  h_{02}\\
                               \end{array}
                             \right]
\end{equation}
for all $c_1,\,c_2\in\dC$. Then, the adjoint operator $K^*:\calH^2\rightarrow\dC^2$ is given by \eqref{e-4-32} and is
\begin{equation}\label{e-ex1-74}
    K^* h=K^*\left[
                               \begin{array}{c}
                               c_1   h_{01} \\
                                 c_2  h_{02}\\
                               \end{array}
                             \right]=\left[
                               \begin{array}{c}
                               (h,c_1   h_{01}) \\
                                (h, c_2  h_{02})\\
                               \end{array}
                             \right]=\left[
                               \begin{array}{c}
                                 c_1 \\
                                 c_2 \\
                               \end{array}
                             \right].
\end{equation}
Following \eqref{e25-I}, \eqref{e-25-J}, and \eqref{e-34-J} we set
\begin{equation}\label{e6-J}
    J_d=I=\left[
          \begin{array}{cc}
            1 & 0 \\
            0 & 1 \\
          \end{array}
        \right], \quad J_m=\left[
          \begin{array}{cc}
            1 & 0 \\
            0 & -1 \\
          \end{array}
        \right], \quad J_a=-I=\left[
          \begin{array}{cc}
            -1 & 0 \\
            0 & -1 \\
          \end{array}
        \right].
\end{equation}

We are constructing three L-systems $\Theta_d$, $\Theta_m$, and $\Theta_a$ of the forms \eqref{e4-27-d}, \eqref{e-4-34}, and \eqref{e-4-36-a}
\begin{equation}\label{e-ex1-75}
    \Theta_d= \begin{pmatrix} T_d&K&\ J_d\cr  \calH^2 & &\dC^2\cr \end{pmatrix},\;
    \Theta_m= \begin{pmatrix} T_m&K&\ J_m\cr  \calH^2 & &\dC^2\cr \end{pmatrix},\;
    \Theta_a= \begin{pmatrix} T_m&K&\ J_a\cr  \calH^2 & &\dC^2\cr \end{pmatrix},
\end{equation}
where operators  $T_d$, $T_m$,  $T_a$, $J_d$, $J_m$, $J_a$,  and $K$ are defined by \eqref{e-ex-63}--\eqref{e6-J}. Note that all three L-systems above share the same channel operator $K$ of the form \eqref{e-ex1-73}.

Using \eqref{e4-28-d}, \eqref{e-4-35}, and \eqref{e-4-27-a}  we have
\begin{equation}\label{e-ex1-76}
    \begin{aligned}
   W_{\Theta_d} (z)&=I-2iK^\ast (T_d-zI)^{-1}KI=\left[
          \begin{array}{cc}
            \frac{-i-z}{i-z} & 0 \\
            0 & \frac{i+z}{{-i+z}} \\
          \end{array}
        \right]=\frac{z+i}{{z-i}}I,\\
     W_{\Theta_m} (z)&=I-2iK^\ast (T_m-zI)^{-1}KJ_m=\left[
          \begin{array}{cc}
            \frac{z+i}{z-i} & 0 \\
            0 & \frac{z-i}{z+i} \\
          \end{array}
        \right],\\
     W_{\Theta_a} (z)&=I-2iK^\ast (T_a-zI)^{-1}K J_a=\left[
          \begin{array}{cc}
            \frac{z-i}{z+i} & 0 \\
            0 & \frac{z-i}{z+i} \\
          \end{array}
        \right]=\frac{z-i}{{z+i}}I.
        \end{aligned}
\end{equation}
The corresponding impedance function is easily found using \eqref{e4-29-d}, \eqref{e-4-36}, and \eqref{e-4-36-aa}
\begin{equation}\label{e-ex1-77}
    \begin{aligned}
 V_{\Theta_d} (z)&=K^\ast (\RE T_d - zI)^{-1} K=\left[
          \begin{array}{cc}
            -\frac{1}{z} & 0 \\
            0 & -\frac{1}{z} \\
          \end{array}
        \right]=-\frac{1}{z}I,\\
    V_{\Theta_m} (z)&=K^\ast (\RE T_m - zI)^{-1} K=\left[
          \begin{array}{cc}
            -\frac{1}{z} & 0 \\
            0 & -\frac{1}{z} \\
          \end{array}
        \right]=-\frac{1}{z}I,\\
        V_{\Theta_a} (z)&K^\ast (\RE T_a - zI)^{-1} K=\left[
          \begin{array}{cc}
            -\frac{1}{z}  & 0 \\
            0 & -\frac{1}{z}  \\
          \end{array}
        \right]=-\frac{1}{z}I.
         \end{aligned}
\end{equation}
As one can see $V_{\Theta_d} (z)=V_{\Theta_m} (z)=V_{\Theta_a} (z)$ for all $z\in\dC_\pm\setminus\{0\}$. Also, the c-entropies $\calS_d$, $\calS_m$, and $\calS_a$  of all three L-systems $\Theta_d$, $\Theta_m$, and $\Theta_a$  are found via \eqref{e-40-entropy}, \eqref{e-41-m}, and \eqref{e-46-entropy}, respectively,  and are
\begin{equation}\label{e-ex1-70}
    \calS_d=+\infty ,\quad \calS_m=0 ,\quad \calS_a=-\infty.
\end{equation}
The corresponding dissipation/accumulation coefficients are (see \eqref{e-69-ent-dis}-\eqref{e-70-ent-acc})
\begin{equation}\label{e-ex1-71}
    \calD_d=1,\quad \calD_m=0 ,\quad \calA_a=1.
\end{equation}

\subsection*{Example 2}\label{ex-2}

Following the steps of Example 1 we let $\lambda_0=1+i$ and consider the linear operators $T_d$, $T_m$, and $T_a$ of the forms \eqref{e4-22-d}, \eqref{e-d-4-30}, and \eqref{e-d-4-33}. Then
\begin{equation}\label{e-ex2-74}
    T_d h=\left[
          \begin{array}{cc}
            1+i & 0 \\
            0 & 1+i \\
          \end{array}
        \right]
     h,\; T_m h=\left[
          \begin{array}{cc}
            1+i & 0 \\
            0 & 1-i \\
          \end{array}
        \right]
     h,\; T_a h=\left[
          \begin{array}{cc}
            1-i & 0 \\
            0 & 1-i \\
          \end{array}
        \right],
\end{equation}
where $h=\left[
                               \begin{array}{c}
                                 h_1 \\
                                 h_2 \\
                               \end{array}
                             \right]   \in\calH^2$.

We are constructing three L-systems $\Theta_d$, $\Theta_m$, and $\Theta_a$ of the forms \eqref{e4-27-d}, \eqref{e-4-34}, and \eqref{e-4-36-a}
\begin{equation}\label{e-ex2-75}
    \Theta_d= \begin{pmatrix} T_d&K&\ J_d\cr  \calH^2 & &\dC^2\cr \end{pmatrix},\;
    \Theta_m= \begin{pmatrix} T_m&K&\ J_m\cr  \calH^2 & &\dC^2\cr \end{pmatrix},\;
    \Theta_a= \begin{pmatrix} T_m&K&\ J_a\cr  \calH^2 & &\dC^2\cr \end{pmatrix},
\end{equation}
where operators  $T_d$, $T_m$,  $T_a$, $J_d$, $J_m$, $J_a$,  and $K$ are defined by \eqref{e-ex2-74}, \eqref{e-ex1-73}--\eqref{e6-J}.

Using \eqref{e4-28-d}, \eqref{e-4-35}, and \eqref{e-4-27-a}  we have
\begin{equation}\label{e-ex2-76}
    \begin{aligned}
           W_{\Theta_d} (z)&=I-2iK^\ast (T_d-zI)^{-1}KI=\left[
          \begin{array}{cc}
            \frac{1-i-z}{1+i-z} & 0 \\
            0 & \frac{1+i+z}{1-i+z} \\
          \end{array}
        \right],\\
     W_{\Theta_m} (z)&=I-2iK^\ast (T_m-zI)^{-1}KJ_m=\left[
          \begin{array}{cc}
            \frac{1-i-z}{1+i-z} & 0 \\
            0 & \frac{1+i-z}{1-i-z} \\
          \end{array}
        \right],\\
     W_{\Theta_a} (z)&=I-2iK^\ast (T_a-zI)^{-1}K J_a=\left[
          \begin{array}{cc}
            \frac{1+i-z}{1-i-z} & 0 \\
            0 & \frac{1-i+z}{1+i+z} \\
          \end{array}
        \right].
        \end{aligned}
\end{equation}
The corresponding impedance function is easily found using \eqref{e4-29-d}, \eqref{e-4-36}, and \eqref{e-4-36-aa}
\begin{equation}\label{e-ex2-77}
    \begin{aligned}
 V_{\Theta_d} (z)&=K^\ast (\RE T_d - zI)^{-1} K=\left[
          \begin{array}{cc}
            \frac{1}{1-z} & 0 \\
            0 & -\frac{1}{1+z} \\
          \end{array}
        \right],\\
    V_{\Theta_m} (z)&=K^\ast (\RE T_m - zI)^{-1} K=\left[
          \begin{array}{cc}
            \frac{1}{1-z} & 0 \\
            0 & \frac{1}{1-z} \\
          \end{array}
        \right]=\frac{1}{1-z}I,\\
        V_{\Theta_a} (z)&K^\ast (\RE T_a - zI)^{-1} K=\left[
          \begin{array}{cc}
            \frac{1}{1-z}  & 0 \\
            0 & -\frac{1}{1+z}  \\
          \end{array}
        \right].
         \end{aligned}
\end{equation}
The c-entropies $\calS_d$, $\calS_m$, and $\calS_a$  of all three L-systems $\Theta_d$, $\Theta_m$, and $\Theta_a$  are found via \eqref{e-40-entropy}, \eqref{e-41-m}, and \eqref{e-46-entropy}, respectively,  and are
\begin{equation}\label{e-ex2-78}
    \calS_d=\ln5 ,\quad \calS_m=0 ,\quad \calS_a=-\ln5.
\end{equation}
The corresponding dissipation/accumulation coefficients are (see \eqref{e-69-ent-dis}-\eqref{e-70-ent-acc})
\begin{equation}\label{e-ex2-79}
    \calD_d=1-e^{-2\ln5}=\frac{24}{25},\quad \calD_m=0 ,\quad \calA_a=1-e^{-2\ln5}=\frac{24}{25}.
\end{equation}


\end{document}